\documentclass[11pt]{article}
\usepackage{amsmath}
\usepackage{amssymb}
\usepackage{amsthm}
\usepackage{graphicx} 
\usepackage[margin=1in]{geometry}
\usepackage[auth-sc-lg]{authblk}
\usepackage{hyperref}

\newtheorem{theorem}{Theorem}[section]
\newtheorem{proposition}[theorem]{Proposition}
\newtheorem{lemma}[theorem]{Lemma}
\theoremstyle{definition}
\newtheorem{example}[theorem]{Example}
\theoremstyle{remark}
\newtheorem{remark}[theorem]{Remark}

\newcommand{\cay}{\mathrm{Cay}}

\newcommand{\ul}[1]{\underline{#1}}

\title{Directed strongly regular graphs \\
from groups, loops and quasigroups}
\author[1]{\v Stefan Gy\"urki\footnote{e-mail of the corresponding author: {\tt stefan.gyurki@stuba.sk} (\v S. Gy\"urki)}}
\author[2]{Mikhail Klin}
\affil[1]{Faculty of Civil Engineering}
\affil[ ]{Slovak University of Technology}
\affil[ ]{Radlinského 11, 810 05 Bratislava,  Slovak Republic}
\affil[ ]{ }
\affil[2]{Department of Mathematics}
\affil[ ]{Ben-Gurion University of the Negev}
\affil[ ]{84105 Beer Sheva, Israel}
\date{}

\begin{document}

\maketitle

\hrule

\begin{abstract}
We introduce four infinite families of directed strongly regular graphs of
orders $2n^2$ and $3n^2$. The constructions are described in terms of
groups, quasigroups, loops and their Latin squares. Two preliminary Cayley
digraph constructions over wreath products are extended to arbitrary
quasigroups and loops, yielding directed strongly regular graphs with
parameters
$(2n^2,3n-2,2n-1,n-1,3),
(2n^2,4n-2,2n+2,n+2,6),
(3n^2,4n-2,2n,n,4),
(3n^2,6n-2,2n+6,n+6,10).
$
\smallskip

\end{abstract}

\hrule

\section{Introduction}

Directed strongly regular graphs (DSRGs) were introduced by Duval
\cite{du} as a directed analogue of strongly regular graphs. Although a
number of constructions are known, many feasible parameter sets remain
unresolved, and general constructions that do not depend on finite-field
structure are therefore useful.

Several sporadic DSRGs and new realizable parameter sets were previously
obtained by the authors through computer algebra experimentation
\cite{gk}. The present paper takes a complementary approach: instead of
searching for individual examples, it develops uniform constructions
producing infinite families of DSRGs from groups, quasigroups, loops and
their Latin squares.

The purpose of this paper is to give four such constructions. We begin
with two Cayley digraph constructions over wreath products involving an
arbitrary finite group. Their adjacency rules depend only on the Latin
square properties of a group table, which allows us to replace the group
by a quasigroup or a loop. This produces two families on $2n^2$ vertices.
Using three copies of the same $n\times n$ grid instead of two gives two
further families on $3n^2$ vertices. All four constructions work for every
admissible order $n$, without a prime-power hypothesis.

The paper is organized as follows. Section~2 recalls the required facts
about DSRGs, Cayley digraphs, Latin squares, quasigroups and loops.
Section~3 gives the two group-based Cayley constructions. Sections~4 and~5
present and verify the quasigroup and loop constructions on $2n^2$ and
$3n^2$ vertices, respectively. A short summary closes the paper.

\section{Preliminaries}

\subsection{General Concepts}

A \emph{simple graph} $\Gamma$ is a pair $(V,E)$, where $V$ is a finite set of \emph{
vertices}, and $E$ is a set of 2-subsets of $V$ which are called \emph{edges}. 
A \emph{directed graph} (briefly digraph) $\Gamma$ is a pair $(V,R)$ where $V$ is the 
set of \emph{vertices} and $R$ is a binary relation on $V$, that is a subset of the set $V^2$
of all ordered pairs of elements in $V$. 
The pairs in $R$ are called \emph{directed arcs} or \emph{darts}.
The vertex set of $\Gamma$ is denoted by $V(\Gamma)$ and the dart set is denoted by $R(\Gamma)$. 

For any finite group $H$ the {\it group ring} $\mathbb ZH$ is defined as the
set of all formal sums of elements of $H$, with coefficients from $\mathbb Z$. 
Let $X$ denote a non-empty subset of $H$. The
element $\sum_{x\in X}x$ in $\mathbb ZH$ is called a 
{\it simple quantity}, and it is denoted as $\ul{X}$. Suppose now that $e\notin X$,
where $e$ is the identity element of the group $H$.
Then the digraph $\Gamma=\cay(H,X)$ with vertex set $H$ and dart set $\{(x,y):
x,y\in H, yx^{-1}\in X\}$ is called the {\it Cayley digraph over $H$ with respect to~$X$.} 

\subsection{Strongly Regular Graphs}

An undirected graph $\Gamma$ is called \emph{regular} of valency $k$, or \emph{$k$-regular},
if each vertex is incident to the same number $k$ of edges.
In the terms of adjacency matrix of a graph we can equivalently say that 
$\Gamma$ is $k$-regular if and only if for its adjacency matrix $A=A(\Gamma)$ the equation
$AJ=JA=kJ$ holds, where $J$ is the all-one matrix.  
A simple regular graph with valency $k$ is said to be \emph{strongly regular}
(SRG, for short) if there exist integers $\lambda$ and $\mu$ such that for each edge $\{u,v\}$ the number
of common neighbours of $u$ and $v$ is exactly $\lambda$; while for each non-edge
$\{u,v\}$ the number of common neighbours of $u$ and $v$ is equal to $\mu$. 
Previous condition can be rewritten equivalently into the equation $A^2=kI+\lambda A+\mu(J-I-A)$
using the adjacency matrix of $\Gamma$.
The quadruple $(n,k,\lambda,\mu)$ is called the \emph{parameter set} of an SRG $\Gamma$. 

The class of strongly regular graphs is well-studied in AGT,
and there are known numerous constructions giving infinite series of them.
As an example of an infinite family of SRGs we show the so-called 
\emph{square lattice graphs}.

\begin{example}
For any integer $n>1$ denote $\mathbb Z_n$ the cyclic group of order $n$, and let us take the 
set $V(\Gamma)=\mathbb Z_n\times\mathbb Z_n$ as the set of vertices of a graph $\Gamma$. 
Two different vertices $(x_1,y_1)$ and $(x_2,y_2)$ will be adjacent if and only if $x_1=x_2$ or
$y_1=y_2$. Graph $\Gamma$ is usually denoted as $L_2(n)$ and it is a SRG with
parameters $(n^2, 2n-2, n-2, 2)$. 
\end{example}

\subsection{Directed Strongly Regular Graphs}

A possible generalization of the notion of SRGs for directed graphs was given by Duval~\cite{du}.
While the family of SRGs has been well-studied in AGT cf. \cite{bv}, 
the directed version has received less attention.
 
A \emph{directed strongly regular graph} (DSRG) with parameters $(n,k,t,\lambda,\mu)$ is a
regular directed graph on $n$ vertices with valency $k$, such that every vertex is incident with~$t$
undirected edges, and the number of paths of length 2 directed from a vertex $x$ to another vertex $y$ is
$\lambda$, if there is an arc from $x$ to $y$, and $\mu$ otherwise. 
In particular, a DSRG with $t=k$ is an SRG, and a DSRG with $t=0$ is a doubly regular tournament.
Throughout the paper we consider only DSRGs satisfying $0<t<k$, which are called \emph{genuine}
DSRGs. 

The adjacency matrix $A=A(\Gamma)$ of a DSRG with parameters $(n,k,t,\lambda,\mu)$, 
 satisfies $AJ=JA=kJ$ and $A^2=tI+\lambda A+\mu(J-I-A)$. 

\begin{remark}
In this paper we prefer for DSRG's 5-tuple of parameters in the order $(n,k,t,\lambda,\mu)$,
though in several other papers the order $(n,k,\mu,\lambda,t)$ is used. 
\end{remark}

\begin{proposition}[\cite{du}]
If $\Gamma$ is a DSRG with parameter set $(n,k,t,\lambda,\mu)$ and adjacency matrix $A$, 
then the complementary graph $\bar\Gamma$ is a DSRG with parameter set $(n,\bar k,\bar t,
\bar\lambda,\bar\mu)$ with adjacency matrix $\bar A=J-I-A$, where 
\begin{eqnarray*}
\bar k &=& n-k-1 \\
\bar t &=& n-2k+t-1 \\
\bar \lambda & =& n-2k+\mu-2 \\
\bar \mu &=& n-2k+\lambda.
\end{eqnarray*}
\label{prop2}
\end{proposition}

\begin{remark}
Proposition \ref{prop2} allows us to restrict our search for the DSRGs with $2k<n$, 
due to complementation, and clearly a discovery of a DSRG with new parameter set
implies a discovery of a DSRG on the complementary parameter set. 
As a consequence, throughout the paper we display just the parameter sets satisfying $2k<n$. 
\end{remark}

For a directed graph $\Gamma$ let $\Gamma^T$ denote the digraph obtained by reversing
all the darts in $\Gamma$. Then $\Gamma^T$ is called the \emph{reverse} of $\Gamma$. In other words,
if $A$ is the adjacency matrix of~$\Gamma$, then $A^T$ is the adjacency matrix of $\Gamma^T$.

The following proposition was observed by Ch. Pech, and presented in \cite{km}, see also \cite{kp}:

\begin{proposition}[\cite{km}]
Let $\Gamma$ be a DSRG. Then the graph $\Gamma^T$ is also a DSRG with the same parameter set. 
\label{prop1}
\end{proposition}

We say that two DSRGs $\Gamma_1$ and $\Gamma_2$ are \emph{equivalent}, if $\Gamma_1\cong\Gamma_2$, 
or $\Gamma_1\cong\Gamma_2^T$, or $\Gamma_1\cong\bar\Gamma_2$, or $\Gamma_1\cong\bar\Gamma_2^T$;
otherwise they are called \emph{non-equivalent}. (In other words, $\Gamma_1$ is equivalent to 
$\Gamma_2$ if and only if $\Gamma_1$ is isomorphic to $\Gamma_2$ or to a graph, obtained from
$\Gamma_2$ via reverse and complementation.) 
From our point of view the interesting DSRGs are those which are non-equivalent. 

The parameters $n,k,t,\lambda,\mu$ are not independent. Relations, which have to be satisfied 
for such parameter sets are usually called \emph{feasibility conditions}. 
Most important, and, in a sense, basic conditions are the following (for their proof see \cite{du}):

\begin{equation}
k(k+\mu-\lambda)=t+(n-1)\mu. \label{c1}
\end{equation}

There exists a positive integer $d$ such that:
\begin{eqnarray}
d^2&=&(\mu-\lambda)^2+4(t-\mu) \label{c2}\\
d &\mid& (2k-(\mu-\lambda)(n-1)) \label{c3}\\
n-1 &\equiv& \frac{2k-(\mu-\lambda)(n-1)}{d}  \pmod 2 \label{c4}\\
n-1 &\geq &\left|\frac{2k-(\mu-\lambda)(n-1)}{d}\right|.\label{c5}
\end{eqnarray}

Further:
\[
\begin{array}{rcccl}
0 & \leq & \lambda < t & < & k \\
0 & < & \mu \leq t & < & k \\
-2(k-t-1) & \leq & \mu  -  \lambda & \leq & 2(k-t).  
\end{array}
\]

We have to mention that for a feasible parameter set 
it is not guaranteed that a DSRG with that parameter set does exist. 
A feasible parameter set for which at least one DSRG $\Gamma$ exists is called \emph{realizable},
otherwise \emph{non-realizable}.
The smallest example of a non-realizable parameter set is $(14,5,4,1,2)$, 
what was shown in \cite{km}.

Known constructions, references, and further data 
concerning DSRGs are collected in one place on the homepage of 
A.E.~Brouwer and S.~Hobart, see \cite{ab}.

\subsection{Latin squares, loops and quasigroups}

Let us denote $I_n=\{1,2,\ldots,n\}$.
A \emph{Latin square of order $n$} is a quadruple $(R,C,S;\ell)$ 
where $R,C,S$ are sets of cardinality $n$ (usually identified with $I_n$), 
called \emph{rows, columns} and \emph{symbols} respectively and $\ell$ 
is a mapping $\ell: R\times C\to S$ such that
for any $i\in R$ and $x\in S$ the equation $\ell(i,j)=x$ has unique 
solution $j\in C$, and for any $j\in C$, $x\in S$ the same equation
has a unique solution $i\in R$. In other words, a Latin square of order $n$
is an $n\times n$ array with $n$ different entries, such that each entry 
occurs exactly once in any row and in any column of the array.

Another representation of Latin squares turns out to be useful.
The \emph{ortogonal array} representation of a Latin square $L$ is 
a set of $n^2$ ordered triplets $\{(i,j,\ell(i,j))\,|\, (i,j)\in I_n\times I_n\}$.
It is easy to observe from the definition, that each ordered pair 
$(i,j)$ occurs exactly once in the first two positions, exactly once on the
last two positions and exactly once in the first and third positions. 

When a Latin square $L$ is a Cayley table of a group $H$, then
such a Latin square will be called a \emph{group Latin square}.

A \emph{quasigroup} is a set $Q$ with a binary operation ``$\cdot$'' 
such that for all $a, b\in Q$ the equations $a \cdot x = b$ and $y \cdot a = b$ 
have a unique solution in $Q$. It is easy to see that every Latin square may be 
interpreted as a multiplication table of a quasigroup, and for each quasigroup 
its Cayley table provides a Latin square.

A \emph{loop} $L$ is a quasigroup with an identity element $e \in L$ with the property
$e\cdot x = x\cdot e = x$ for every $x \in L$. Usually, the identity element $e$ is identified with
$1 \in I_n$. Then we may say that each loop naturally defines a \emph{reduced} 
(or \emph{normalized}) Latin square and each reduced Latin square may be interpreted 
as a Cayley table of a loop. Obviously, an associative loop is a group.

\subsection{The standard grid parameter family}

The principal family considered below has parameters
$(2n^2,3n-2,2n-1,n-1,3)$.
Earlier constructions of this family used finite-field or partial-sum-family
methods and therefore assumed that $n$ was a prime power; see \cite{ar, fm}.
The constructions below use only a quasigroup table and work for arbitrary integers $n\geq2$.

\section{Cayley constructions from groups}

We use the following standard group-ring criterion.

\begin{lemma}[\cite{hs,km}]
The Cayley digraph $\cay(H,X)$ is a DSRG with parameter set
$(n,k,t,\lambda,\mu)$ if and only if
\[
\ul{X}^2=t\ul{e}+\lambda\ul{X}
  +\mu(\ul{H}-\ul{e}-\ul{X})
\]
in the group ring $\mathbb ZH$.
\label{lem1}
\end{lemma}

Let us consider the group $H=S_2 \wr K$, the wreath product of the symmetric group $S_2$ 
on two symbols $\{1,2\}$, and an arbitrary group $K$ of order $n$. 
Thus, any element $h$ of $H$ can be represented as $h=(g; k_1,k_2)$, 
where $g\in S_2$, $k_1,k_2\in K$ and the multiplication 
of element $h=(g; k_1,k_2)$ with $h'=(g'; k'_1,k'_2)$ is given by 
$hh'=(gg'; k_1k'_{1^g},k_2k'_{2^g})$. 
(Here $1^g$ and $2^g$ denotes the images of $1$ and $2$
under $g\in S_2$, respectively.)
In the rest of this section we use the following notation for certain subsets of~$H$:
$$A=\bigcup_{\substack{u\in K\\ u\neq e_K}}\{(e;u,u)\},\, 
B=\bigcup_{\substack{u\in K\\ u\neq e_K}}\{(e;e_K,u)\},\,
C=\bigcup_{u\in K}\{(\pi;u,e_K)\},\, D=\bigcup_{u\in K}\{(\pi;u,w)\},$$ 
where $\pi$ is the transposition $(1,2)$ in $S_2$ and $ w$ is an 
arbitrary non-identity element of $K$. 
Similarly, let us put
$$ H_0=\{(e; x,y)\in H: x,y\in K\},\quad H_1=\{(\pi;x,y)\in H: x,y\in K\}.$$

Now, we are ready to formulate our results.

\begin{theorem}\label{thm:3.2}
Let $n>1$ be an integer and $X=A\cup B\cup C$.
Then $\cay(H, X)$ is a DSRG with parameter set $(2n^2,3n-2,2n-1,n-1,3)$. 
\label{thm1}
\end{theorem}

\begin{proof} 
Let us compute:
\[
\ul{X}\cdot\ul{X} = (\ul{A}+\ul{B}+\ul{C})^2=(n-1)\cdot\ul{e}+(n-2)\cdot\ul{A}+(\ul{H_0}-
\ul{A}-\ul{B}-\ul{e})+(\ul{H_1}-\ul{C})+\]
\[
+(\ul{H_0}-\ul{A}-\ul{B}-\ul{e})+(n-1)\cdot\ul{e}+(n-2)\cdot\ul{B}+(\ul{H_1}-\ul{C})+(\ul{H_1}-\ul{C})+
(n-1)\cdot\ul{C}+\ul{H_0}=\]
\[=(2n-1)\cdot\ul{e}+(n-1)\cdot \ul{X}+3\cdot(\ul{H}-\ul{e}-\ul{X}).
\]
\end{proof}

\begin{theorem}\label{thm:3.3}
Let $n>2$ be an integer and $Y=X\cup D=A\cup B\cup C\cup D$. 
Then $\cay(H, Y)$ is a DSRG with parameter set $(2n^2,4n-2,2n+2,n+2,6)$. 
\label{thm2}
\end{theorem}

\begin{proof} 
Using computations from the proof of the previous theorem we are getting:
\[ \ul{Y}\cdot\ul{Y}=(\ul{X}+\ul{D})^2=\ul{X}^2+\ul{X}\cdot\ul{D}+\ul{D}\cdot\ul{X}+
\ul{D}\cdot\ul{D}=(2n-1)\cdot\ul{e}+(n-1)\cdot\ul{X}+
\]
\[
+3\cdot(\ul{H}-\ul{e}-\ul{X})+(\ul{H_1}-\ul{D})+(\ul{H_1}-\ul{D})+\ul{H_0}+
\ul{H_0}+(\ul{H_1}-\ul{D})+(n-1)\cdot\ul{D}+\ul{H_0}=
\]
\[
=(2n+2)\cdot\ul{e}+(n+2)\cdot\ul{Y}+6\cdot(\ul{H}-\ul{e}-\ul{Y}).
\]
\end{proof}

\begin{remark} For $n=2$ the previous construction gives us an undirected
graph of order 8, which is a complement to a perfect matching on 8 vertices.  
\end{remark}

\begin{remark} We would like to notice that if $|K|=n$, then
the digraphs constructed in previous two theorems contain two disjoint induced copies 
of the square lattice graphs of order $n^2$, due to the appearance of sets $A$ and $B$ 
as subsets in the connection set of the considered Cayley graphs. 
\end{remark}

\begin{remark} We refer to the families in Theorems~\ref{thm1}
and~\ref{thm2} as Constructions 1 and 2, respectively.
While in Construction 1 the resulting Cayley graph is
uniquely determined by the group $K$, in Construction 2 we have a freedom
for the choice of element $w$ in the connection set $D$, so for the same
group we can get more than one DSRG.
\end{remark}

\section{Constructions on $2n^2$ vertices}

In this section we generalize our Constructions 1 and 2 which worked with the aid of a group.
We extend our constructions for the case, when the starting object is a loop or a quasigroup. 

Let $L$ be a loop (quasigroup) of order $n$. 
We refer to the symbol appearing in the $x$-th row and $y$-th column as $xy$. 

\subsection{First construction in terms of quasigroups (loops)}

Define a digraph $\Gamma_1$ of order $2n^2$, whose 
vertex set is $V(\Gamma_1)= I_n\times I_n\times\mathbb Z_2$. 
The set $D(\Gamma_1)$ of darts is defined as follows:
\begin{itemize}
\item $(x,y,i)\mapsto(z,y,i)$ for all $i\in\mathbb Z_2$, $x,y,z\in I_n$, $x\neq z$;
\item $(x,y,i)\mapsto(x,z,i)$ for all $i\in\mathbb Z_2$, $x,y,z\in I_n$, $y\neq z$;
\item $(x,y,0)\mapsto(xy,z,1)$ for all $z\in  I_n$;
\item $(x,y,1)\mapsto(z,yx,0)$ for all $z\in  I_n$.
\end{itemize}

\begin{theorem}\label{thm:4.1}
The digraph $\Gamma_1$ is a DSRG with parameter set $(2n^2,3n-2,2n-1,n-1,3)$.
\end{theorem}

\noindent{\bf Proof.}

{\it Regularity:} From the definition immediately follows that each vertex has
out-degree $3n-2$. Also it is not hard to see that the number of undirected 
edges incident to any vertex is $2n-1$. 

{\it Existence of $\lambda$:}
For any two vertices $v,w\in V$ such that there is a dart from $v$ to $w$,
we need to prove that the number of oriented paths of length two from $v$ to $w$ 
is equal to $n-1$. 

\smallskip

Take vertices $v=(x,y,0)$ and $w=(x,z,0)$, $y\neq z$. We need to count how many vertices $(u,r,i)$
there exists such that $(x,y,0)\mapsto(u,r,i)\mapsto(x,z,0)$ is an oriented path 
of length two. 
For $i=0$, $u=x$ clearly all $r\neq y,z$ is giving such a path, therefore we have 
altogether $n-2$ of them. 
For $i=0$, $u\neq x$ necessarily $r=y$, but $y\neq z$ so $u=x$, a contradiction. 
If $i=1$ then $r=yx$, further $x=u(yx)$ and this equation has a unique solution $u$,
giving one good path. 
This means that altogether there are $n-1$ oriented path of length two from
$(x,y,0)$ to $(x,z,0)$. 

The proofs for the cases $v=(x,y,1)$, $w=(x,z,1)$;  $v=(x,y,0)$, $w=(z,y,0)$; 
and $v=(x,y,1)$, $w=(z,y,1)$ could be done similarly. 

\smallskip

Let us take $v=(x,y,0)$ and $w=(z,yx,0)$. For $i=0$ in the case when $u=x$ we are
getting equation $rx=yx$, which has exactly one solution $r$ for fixed $x,y$, 
but clearly, $r=y$ therefore $(u,r,0)=v$ and we do not have a 2-path.
For $i=0$ the assumption $r=y$ leads also to $(u,r,0)=v$. 
Now, if $i=1$, then necessarily $r=yx$, and all $u\neq z$ is giving a good 2-path. 
Thus, the number of oriented 2-path from $v$ to $w$ is $n-1$ in this case. 

Analogously, we can prove that the number of oriented 2-paths from $v=(x,y,1)$
to $w=(xy,z,0)$ is also $n-1$.

\medskip

{\it Existence of $\mu$:}

For any two vertices $v,w\in V$ such that there is no dart from $v$ to $w$,
we need to prove that the number of oriented paths of length two from $v$ to $w$
is equal to $3$. 

\smallskip

Let us take $v=(x,y,0)$ and $w=(\alpha,\beta,0)$, $x\neq \alpha$, $y\neq \beta$. 
We need to count how many vertices $(u,r,i)$ there exists such that 
$(x,y,0)\mapsto(u,r,i)\mapsto(x,z,0)$ is an oriented path of length two. 
For $i=0$ clearly $(x,\beta,0)$ and $(\alpha,y,0)$ are the only two such vertices.
For $i=1$ necessarily $r=yx$, and $u(yx)=\alpha$, which has exactly one solution for $u$.
So altogether we have three oriented 2-paths from $v$ to $w$. 

Similarly, we can see that there are three oriented 2-paths from $v=(x,y,1)$ to 
$w=(\alpha,\beta,1)$, $x\neq \alpha$, $y\neq \beta$.

Now, let us take $v=(x,y,0)$ and $w=(\alpha,\beta,1)$, $\beta\neq yx$. 
For $i=0$, necessarily $u=x$ or $r=y$. In the first case $\beta=rx$, which has
exactly one solution for $r$, while in the latter case $\beta=yu$, which has also
exactly one solution for $u$. For $i=1$ necessarily $r=yx$, and then $\alpha=u$ 
is giving a solution, while $\beta=yx$ gives a contradiction, so altogether there
are three oriented 2-paths from $v$ to $w$.

The case $v=(x,y,1)$, $w=(\alpha,\beta,0)$, $\alpha\neq xy$ can be done similarly. 

\rightline{$\square$}

Since the Cayley tables of quasigroups are Latin squares, and every Latin square
can be interpreted as a Cayley table of a quasigroup, we can provide a description 
of the previous construction purely in the terms of Latin squares.

For arbitrary positive integer $n\geq 2$ let $L$ be a Latin square corresponding to a quasigroup 
of order $n$, $M=\{1,2,\ldots,n\}$ and $N=\{n+1,n+2,\ldots,2n\}$.
Define a digraph with vertex set $V=(M\times N)\cup (N\times M)$. Let us define a few types of edges 
(darts):
$$E_1=\{(a,b)\sim(c,d)\,|\, (a,b),(c,d)\in V, a=c \textrm{ and } b\neq d\},$$ 
$$E_2=\{(a,b)\to (c,d)\,|\, (a,b),(c,d)\in V, a=d \textrm{ or } b=c\},$$
$$E_3=\{(a,b)\sim(c,d)\,|\, (a,b),(c,d)\in M\times N, L_{a,b-n}=L_{c,d-n}\},$$
$$E_4=\{(a,b)\sim(c,d)\,|\, (a,b),(c,d)\in N\times M, L_{b-n,a}=L_{d-n,c}\}.$$

We would like to notice that a digraph $\bar\Gamma$ with vertex set $V$ and edge (dart) set 
$\bar E = E_1 \cup E_2$ is a DSRG with parameter set $(2n^2,2n-1,n,n-1,1)$ and it was constructed 
in \cite{kp} from generalized quadrangles $GQ(1,n-1)$.

\begin{theorem}
The digraph $\Gamma$ with vertex set $V$ and edge (dart) set $E=E_1\cup E_2\cup E_3\cup E_4$
is a DSRG with parameters $(2n^2,3n-2,2n-1,n-1,3)$ and is equivalent to the graph
obtained in Theorem \ref{thm:4.1} with the aid of the same quasigroup.
\end{theorem}
 
\medskip

\subsection{Second construction in terms of quasigroups (loops)}

Define a digraph $\Gamma_3$ of order $2n^2$, whose 
vertex set is $V(\Gamma_3)= I_n\times I_n\times\mathbb Z_2$. 
Let $c$ be any non-identity element of a loop $L$. 
The set $D(\Gamma_3)$ of darts is defined as follows:
\begin{itemize}
\item $(x,y,i)\mapsto(z,y,i)$ for all $i\in\mathbb Z_2$, $x,y,z\in I_n$, $x\neq z$;
\item $(x,y,i)\mapsto(x,z,i)$ for all $i\in\mathbb Z_2$, $x,y,z\in I_n$, $y\neq z$;
\item $(x,y,0)\mapsto(xy,z,1)$ for all $z\in  I_n$;
\item $(x,y,1)\mapsto(z,yx,0)$ for all $z\in  I_n$;
\item $(x,y,0)\mapsto(c(xy),z,1)$ for all $z\in  I_n$;
\item $(x,y,1)\mapsto(z,(yx)c,0)$ for all $z\in  I_n$.
\end{itemize}

\begin{theorem}
The digraph $\Gamma_3$ is a DSRG with parameter set $(2n^2,4n-2,2n+2,n+2,6)$.
\end{theorem}

\noindent{\bf Proof.}
The proof of this theorem can be executed analogously as for Theorem \ref{thm:4.1}.

\rightline{$\square$}

\section{Constructions on $3n^2$ vertices}

In this section we are providing one more generalization
of our previous constructions. This time we increase the number
of copies of the ``$n\times n$ grid'' from two to three.

\subsection{Third construction in terms of quasigroups (loops)}

Define a digraph $\Gamma_2$ of order $3n^2$, whose vertex set is 
$V=I_n\times I_n\times\mathbb Z_3$. The set $D$ of darts is defined 
as follows:
\begin{itemize}
\item $(x,y,i)\mapsto(z,y,i)$ for all $i\in\mathbb Z_3$, $x,y,z\in I_n$, $x\neq z$;
\item $(x,y,i)\mapsto(x,z,i)$ for all $i\in\mathbb Z_3$, $x,y,z\in I_n$, $y\neq z$;
\item $(x,y,i)\mapsto(xy,z,i+1)$ for all $i\in\mathbb Z_3$, and $z\in  I_n$;
\item $(x,y,i)\mapsto(z,yx,i-1)$ for all $i\in\mathbb Z_3$, and $z\in  I_n$.
\end{itemize}

\begin{theorem}
The digraph $\Gamma_2$ is a DSRG with parameter set $(3n^2,4n-2,2n,n,4)$.
\end{theorem}

\subsection{Fourth construction in terms of quasigroups (loops)}

Define a digraph $\Gamma_4$ of order $3n^2$, whose 
vertex set is $V(\Gamma_4)= I_n\times I_n\times\mathbb Z_3$. 
Let $c$ be any non-identity element of the loop $L$. 
The set $D(\Gamma_4)$ of darts is defined as follows:
\begin{itemize}
\item $(x,y,i)\mapsto(z,y,i)$ for all $i\in\mathbb Z_3$, $x,y,z\in I_n$, $x\neq z$;
\item $(x,y,i)\mapsto(x,z,i)$ for all $i\in\mathbb Z_3$, $x,y,z\in I_n$, $y\neq z$;
\item $(x,y,i)\mapsto(xy,z,i+1)$ for all $z\in  I_n$;
\item $(x,y,i)\mapsto(z,yx,i-1)$ for all $z\in  I_n$;
\item $(x,y,i)\mapsto(c(xy),z,i+1)$ for all $z\in  I_n$;
\item $(x,y,i)\mapsto(z,(yx)c,i-1)$ for all $z\in  I_n$.
\end{itemize}

\begin{theorem}
The digraph $\Gamma_4$ is a DSRG with parameter set $(3n^2,6n-2,2n+6,n+6,10)$.
\end{theorem}

The proofs of the theorems in this section can be done similarly like in the previous section
using easy counting arguments on directed paths.

\section{Concluding remarks}

The four constructions above yield DSRGs from quasigroups and loops of
arbitrary admissible order. The first two families have order $2n^2$ and
parameters
\[
(2n^2,3n-2,2n-1,n-1,3)
\quad\hbox{and}\quad
(2n^2,4n-2,2n+2,n+2,6),
\]
while the remaining two have order $3n^2$ and parameters
\[
(3n^2,4n-2,2n,n,4)
\quad\hbox{and}\quad
(3n^2,6n-2,2n+6,n+6,10).
\]
The constructions use only the defining equations of a quasigroup, together
with an identity element and a chosen nonidentity element where explicitly
required. Consequently, they do not rely on associativity or on the
existence of a finite field.

For completeness, Table~\ref{tab:small-parameters} lists all parameter sets arising directly from our four constructions for which the order of
the resulting graph does not exceed \(110\). This upper bound was chosen to facilitate comparison with the table of directed strongly regular graphs maintained by Andries E. Brouwer \cite{ab}.

\begin{table}[ht]
\centering
\small
\setlength{\tabcolsep}{4pt}
\caption{Parameter sets obtained from Constructions 1--4
with graph order at most \(110\).}
\label{tab:small-parameters}
\begin{tabular}{|c|c|c|c|c|}
\hline
\(n\) & Construction 1 & Construction 2
      & Construction 3 & Construction 4\\
\hline
2 & \((8,4,3,1,3)\)
  & \((8,6,6,4,6)\)
  & \((12,6,4,2,4)\)
  & \((12,10,10,8,10)\)\\

3 & \((18,7,5,2,3)\)
  & \((18,10,8,5,6)\)
  & \((27,10,6,3,4)\)
  & \((27,16,12,9,10)\)\\

4 & \((32,10,7,3,3)\)
  & \((32,14,10,6,6)\)
  & \((48,14,8,4,4)\)
  & \((48,22,14,10,10)\)\\

5 & \((50,13,9,4,3)\)
  & \((50,18,12,7,6)\)
  & \((75,18,10,5,4)\)
  & \((75,28,16,11,10)\)\\

6 & \((72,16,11,5,3)\)
  & \((72,22,14,8,6)\)
  & \((108,22,12,6,4)\)
  & \((108,34,18,12,10)\)\\

7 & \((98,19,13,6,3)\)
  & \((98,26,16,9,6)\)
  & 
  & \\
\hline
\end{tabular}
\end{table}

\medskip

The very early results from this project have been presented at various
events on algebraic theory. Both authors attended workshops and conferences
on topics suitable for giving a talk on the achieved results. 
Here we list the most remarkable of them. 

\begin{itemize}
\item \emph{Modern Trends in Algebraic Graph Theory}, Villanova (PA), USA, June 2014.
\item \emph{Second Joint International Meeting of the
Israel Mathematical Union and the American Mathematical Society}, 
Tel Aviv, Israel, June 2014. 
\item \emph{Summer School on General Algebra and Ordered Sets}, 
Star\'a Lesn\'a, Slovakia, September 2014.
Slides available on:
\url{http://im.saske.sk/ssaos2014/talks/monday/Klin.pdf}
\end{itemize}

The present work also has its roots in an earlier attempt to bring directed strongly regular graphs into the developing framework of algebraic combinatorics and algebraic graph theory. In his lecture \emph{Automorphism groups of circulant graphs}, delivered at the conference \emph{Applicable Algebra} at the Mathematisches Forschungsinstitut Oberwolfach in January 1994, the second author discussed connections with the pioneering work of Duval on directed strongly regular graphs \cite{ko}. To the best of our knowledge, this presentation was among the earliest attempts to relate Duval's theory explicitly to emerging methods and perspectives in algebraic combinatorics and algebraic graph theory.

The present paper is a substantially revised and shortened version of an unpublished ma\-nu\-script from 2014 that was circulated among only a few colleagues. The current version focuses on rigorously established results and outlines of their proofs, sometimes too brief, 
while the conjectures included in the original manuscript have been omitted.

A much more detailed version of \cite{km} was created in 1997, seven years earlier, before the publication of the shorter article. Its copy, as a Preprint, was circulated between colleagues and is available from the authors on a special request.

\section*{Acknowledgements}

\indent This research was supported by the APVV Research Grants under number 22-0005, 23-0076 and
also by VEGA Research Grants 1/0069/23 and 1/0011/25.
The authors gratefully acknowledge Roman Nedela for his keen interest in and valuable support of their research on directed strongly regular graphs.
The authors are very pleased to thank Andries E. Brouwer
for his ongoing interest to our joint research,
and especially, to this text.

\end{document}